\documentclass[12pt]{article}
\usepackage{amssymb}
\usepackage{amsmath}
\usepackage{graphicx}
\usepackage{enumerate}
\newenvironment{proof}[1][Proof]{\noindent\textit{#1.} }{\hfill  \rule{0.5em}{0.5em}}

\newtheorem{theorem}{Theorem}
\newtheorem{lm}{Lemma}
\newtheorem{thm}[theorem]{Theorem}
\newtheorem{pro}{Proposition}
\newtheorem{exmp}{Example}
\newtheorem{rmk}{Remark}
\newtheorem{coro}{Corollary}

\begin{document}

\title{Continuous solutions of an iterative equation \\
with multiplication
}

\author{
Chaitanya Gopalakrishna\,$^a$,~~
Murugan Veerapazham\,$^a$,
\vspace{3mm}\\
Suyun Wang\,$^b$,~~
Weinian Zhang\,$^c$
\vspace{3mm}\\
$^a${\small Department of Mathematical and Computational Sciences,}
\\
{\small National Institute of Technology Karnataka,}
\\
{\small  Surathkal, Mangalore-575 025, India}
\\
$^b${\small School of Mathematics, Lanzhou City University,}
\\
{\small 11 Jiefang Road, Lanzhou, Gansu 730070, P. R. China}
\\
$^c${\small School of Mathematics, Sichuan University,}
\\
{\small Chengdu, Sichuan 610064, P. R. China}

 \vspace{0.2cm}\\
{\small cberbalaje@gmail.com (CG),~~murugan@nitk.edu.in (MV),}\\
{\small wangsy1970@163.com (SW),~~matzwn@126.com (WZ).}       
   
}

\date{}


\maketitle

\begin{abstract}
Iterative equation is an equality with an unknown function and its iterates. 
There were not found a result on iterative equations with multiplication of iterates of the unknown function on $\mathbb{R}$.
In this paper we use an exponential function
to reduce the equation in conjugation to the well-known form of polynomial-like iterative equation, but we encountered
two difficulties: the reduction restricts our discussion of the equation to $\mathbb{R}_+$; 
the reduced polynomial-like iterative equation is defined on the whole $\mathbb{R}$ but
known results were given on comapct intervals.
We revisit the  polynomial-like iterative equation on the whole $\mathbb{R}$ and 
give existence, uniqueness, stability and construction of continuous solutions of our equation on $\mathbb{R}_+$.
Then we technically extend our solutions from $\mathbb{R}_+$ to $\mathbb{R}_-$.

\vskip 0.2cm

{\bf Keywords:}
functional equation; iteration; nonlinear combination; contraction principle.
\vskip 0.2cm

{\bf MSC(2010):}
primary 39B12; secondary 47J05.
\end{abstract}




\baselineskip 16pt
\parskip 10pt


\section{Introduction}

Consider a map $f: E\to E$ on a nonempty set $E$. Its $n$-th order iterates, denoted by
$f^n$,
is defined recursively by $f^0={\rm id}$, the identity map on $E$,  and $f^{n+1}=f\circ f^n$.
Being an important operation in the present era of informatics, iteration gets more and more attractive to researchers
and
attentions (see the book \cite{Kuczma} and the survey \cite{Baron-Jarczyk} for example)
were paid to those functional equations involving iteration, called
{\it iterative equations}.
The general form of such equations can be presented as
$$
\Phi(f(x), f^2(x),..., f^n(x))=F(x),
$$
where $F$ and $\Phi$ are given functions and $f$ is unknown.
Some special cases of this equation,
for example,
iterative root problem (\cite{Kuczma1968,Targonski}), dynamics of a quadratic mapping (\cite{greenfield})
 and Feigenbaum's equation related to period doubling bifurcations (\cite{Mc}),
are interesting topics in dynamical systems.

Although there can be found several papers \cite{Mu-Su, Si,Wang-Si} on the general Lipschitzian $\Phi$,
more efforts were still made to the basic form
\begin{eqnarray}\label{3}
a_1f(x)+\cdots+ a_n f^n(x)= F(x)
\end{eqnarray}
with $\Phi$ in a linear combination, called the polynomial-like iterative equation,
for more concrete properties.
Continuous solutions, differentiable solutions, convex solutions and decreasing solutions, and equivariant solutions
of \eqref{3} are discussed in \cite{Xu-Zhang,zhang1990, Bing-Weinian, Zhang-Edinb} respectively.
It is also interesting to discuss $\Phi$ of nonlinear combination, but
no results on nonlinear combination are found except for \cite{Zdun-Zhang}
on the unit circle.

In this paper, we consider the iterative equation with multiplication
\begin{eqnarray}
\begin{array}{ll}
(g(x))^{\alpha_1} (g^{2}(x))^{\alpha_2}\cdots(g^{n}(x))^{\alpha_n} =G(x),
\end{array}
\label{1}
\end{eqnarray}
i.e., $\Phi(u_1, u_2, \ldots,u_n)=\prod_{k=1}^{n}u_k^{\alpha_k}$,
where $G$ is given and $g$ is unknown.
Unlike those \cite{Mu-Su, Si,Wang-Si,Si-Zhang,zhang1988,zhang1989} on compact intervals,
our work to (\ref{1}) is concentrated to solving \eqref{3} on the whole $\mathbb{R}$.
Our strategy is to restrict our discussion of \eqref{1} on $\mathbb{R}_+:=(0,+\infty)$
and use an exponential function
to reduce in conjugation to the well-known form of polynomial-like iterative equation (\ref{3}) on the whole $\mathbb{R}$.
Note that all found results on the polynomial-like iterative equation
are given either on a compact interval or near a fixed point, none of which
can be applied to our case.
In this paper we generally discuss a polynomial-like iterative equation on the whole $\mathbb{R}$
and use obtained result to give solutions of equation (\ref{1}) on $\mathbb{R}_+$ and $\mathbb{R}_-:=(-\infty, 0)$. Our approach here is twofold.
First, using the Banach contraction principle, we give sufficient conditions for
existence and uniqueness of solutions for \eqref{1}. We also prove that the obtained solution depends on $G$ continuously.
Then, using the second method, we construct its solutions sewing piece by piece as done in \cite{Xu-Zhang,Zhang-Xu-Zhang}.


\section{Preliminaries}

Let
$\mathcal{C}_b(\mathbb{R})$ (resp. $\mathcal{C}_b(\mathbb{R}_+)$) consist of all bounded continuous self-maps of $\mathbb{R}$ (resp. $\mathbb{R}_+$). Then $\mathcal{C}_b(\mathbb{R})$
is a Banach space with the uniform norm $\|\cdot\|$, defined by
$$\|f\|:=\sup \{|f(x)|: x\in \mathbb{R}\}.$$
Consider $g$ on $(0,+\infty)$. We can use the exponential map $\psi(x)=e^x$ to conjugate $g$ to a map on the whole $\mathbb{R}$, i.e.,
let $f(x):=\log g(e^x)$, a map from $\mathbb{R}$ into $\mathbb{R}$
(one-to-one if $g$ is one-to-one),
and reduces equation (\ref{1}) to the polynomial-like one
\begin{eqnarray}
\alpha_1 f(x) +\alpha_2 f^2(x)+\cdots +\alpha_nf^n(x)= F(x), ~~~~~~    x\in \mathbb{R},
\label{(*)}
\end{eqnarray}
where $F(x):=\log G(e^x)$.

\begin{pro}\label{P3}
The map $g$ is a solution (resp. unique solution) of \eqref{1} in $\mathcal{X}\subseteq \mathcal{C}_b(\mathbb{R}_+)$
if and only if $f(x):=\psi^{-1}(g(\psi(x)))$
is a solution (resp. unique solution) of \eqref{(*)} in $\mathcal{Y}\subseteq \mathcal{C}_b(\mathbb{R})$, where $\psi(x)=e^x$ and $\mathcal{Y}=\{\psi^{-1}\circ g\circ \psi: g\in \mathcal{X}\}$.
\end{pro}

\begin{proof}
	Let $g$ be a solution of \eqref{1} in $\mathcal{X}$. Since $\psi$ is a homeomorphism of $\mathbb{R}$ onto $\mathbb{R}_+$, clearly $\mathcal{Y} \subseteq \mathcal{C}_b(\mathbb{R})$ and $f\in \mathcal{Y}$.  Also, for each $x\in \mathbb{R}$, we have
	\begin{eqnarray*}
	\sum_{k=1}^{n}\alpha_k f^k(x)&=&\sum_{k=1}^{n}\alpha_k \log g^k(e^x)\\
	&=& \log\left(\prod_{k=1}^{n}(g^k(e^x))^{\alpha_k}\right) \\
	&=&\log G(e^x)=F(x),
	\end{eqnarray*}
implying that $f$ is a solution of \eqref{(*)} on $\mathbb{R}$.  The converse follows similarly. Now, in order to prove the uniqueness, assume that \eqref{1} has a unique solution in $\mathcal{X}$ and
suppose that $f_1, f_2$ are any two solutions of \eqref{(*)} in $\mathcal{Y}$. Then, by ``if'' part of what we have proved above, there exist solutions  $g_1$ and $g_2$ of \eqref{1} in $\mathcal{X}$ such that $f_1=\psi^{-1}\circ g_1\circ \psi$ and $f_2=\psi^{-1}\circ g_2\circ \psi$. By our assumption, we have $g_1=g_2$ and therefore $f_1=f_2$. The proof of converse is similar.
\end{proof}


Consider $g$ on $(-\infty,0)$, we have the following.

\begin{pro}\label{P4}
Let
$\alpha_k\in \mathbb{Z}$ for $1\le k\le n$ such that $\sum_{k=1}^{n}\alpha_k$ is odd.
Then the map $g$ is a solution (resp. unique solution) of \eqref{1} in $\mathcal{X}\subseteq \mathcal{C}_b(\mathbb{R}_-)$
if and only if $h(x):=\psi^{-1}(g(\psi(x)))$
is a solution (resp. unique solution) of the equation
\begin{eqnarray}\label{15}
(h(x))^{\alpha_1} (h^{2}(x))^{\alpha_2}\cdots(h^{n}(x))^{\alpha_n}=H(x)
\end{eqnarray}
in  $\mathcal{Y}\subseteq \mathcal{C}_b(\mathbb{R}_+)$,
where $\psi(x)=-x$, $H(x)=\psi^{-1}(G(\psi(x)))$ and $\mathcal{Y}=\{\psi^{-1}\circ g\circ \psi: g\in \mathcal{X}\}$.
\end{pro}

\begin{proof}
Let $g$ be a solution of \eqref{1} in $\mathcal{X}$.
Since $\psi$ is a homeomorphism of $\mathbb{R}_+$ onto $\mathbb{R}_-$, clearly $\mathcal{Y}\subseteq \mathcal{C}_b(\mathbb{R}_+)$  and $h\in \mathcal{Y}$.
Also, for each $x\in \mathbb{R}_+$ and $k\in \{1,2, \ldots, n\}$, we have $H(x)=-G(-x)$ and $h^k(x)=-g^k(-x)$. Therefore,
\begin{eqnarray*}
\prod_{k=1}^{n}(h^{k}(x))^{\alpha_k}
&=& \prod_{k=1}^{n}(-g^k(-x))^{\alpha_k}
\\
&=&(-1)^{\sum_{k=1}^{n}\alpha_k}\prod_{k=1}^{n} (g^k(-x))^{\alpha_k}
\\
	&=&-\prod_{k=1}^{n} (g^k(-x))^{\alpha_k}
\\
	&=&-G(-x)=H(x)
	\end{eqnarray*}
since
$\sum_{k=1}^{n}\alpha_k$ is odd,
implying that $h$ is a solution of \eqref{15} on $\mathbb{R}_+$.
The converse follows similarly. Further, the proof of uniqueness is similar to that of Proposition \ref{P3}.
\end{proof}

By Proposition \ref{P3}, it suffices to prove existence for \eqref{(*)} on the whole $\mathbb{R}$
in order to prove the existence of solution for \eqref{1} on $\mathbb{R}_+$.
Further, in order to extend the solutions from $\mathbb{R}_+$ to its closure,
we require the continuity of $g$ and $G$ at $0$, i.e., we require the necessary conditions $\lim_{x\to -\infty} F(x)=\log G(0)$,
$\lim_{x\to -\infty} f^k(x)=\log g^k(0)$ for all
$1\le k\le n$ and
$$
\sum_{k=1}^{n}\alpha_k \log g^k(0)= \log G(0).
$$
These conditions can indeed be satisfied
if $\sum_{k=1}^{n}\alpha_k=1$, $G|_{\mathcal{R}(G)}={\rm id}$, $0\notin \mathcal{R}(G)$, and $G\equiv g$ on $[0,+\infty)$.

Let $I:=[a,b]$ and $J:=[c,d]$ be compact intervals in $\mathbb{R}$ and $\mathbb{R}_+$ respectively with non-empty interiors.
Let
$\mathcal{C}(\mathbb{R}, I)$ (resp. $\mathcal{C}(I,I)$) be the set of all continuous maps of $\mathbb{R}$ (resp. $I$) into $I$.
Similarly we define $\mathcal{C}(\mathbb{R}_+, J)$ and $\mathcal{C}(J,J)$.
For any $f$ in $\mathcal{C}(\mathbb{R}, I)$ or $\mathcal{C}(I,I)$,
let
$$\|f\|_I:=\sup \{|f(x)|: x\in I\},$$ and for any
$g$ in $\mathcal{C}(\mathbb{R}_+, J)$ or $\mathcal{C}(J,J)$,
let
$$\|g\|_J:=\sup \{|g(x)|: x\in J\}.$$
For any map $f$, let $\mathcal{R}(f)$
denote the range of $f$.
For $M, \delta\ge 0$, let
\begin{eqnarray*}
	\mathcal{F}_{I}(\delta,M)&:=&\{f\in \mathcal{C}_b(\mathbb{R}): \mathcal{R}(f)=I, f(a)=a, f(b)=b~\text{and}\\
	& &~~~~~\delta(x-y)\le f(x)-f(y)\le M(x-y), \forall x, y \in I~\mbox{with}~x\ge y\},
\\
	\mathcal{G}_{J}(\delta,M)&:=&\left\{g\in \mathcal{C}_b(\mathbb{R}_+): \mathcal{R}(g)=J, g(c)=c, g(d)=d\right. \text{and}\\
	& &~~~~~\left.\left(\frac{x}{y}\right)^\delta \le \frac{g(x)}{g(y)}\le \left(\frac{x}{y}\right)^M, \forall x, y \in J~\mbox{with}~x\ge y\right\}.
\end{eqnarray*}
Then it can be observed that  $\mathcal{F}_{I}(\delta,M) \subseteq \mathcal{F}_{I}(\delta_1,M_1)$ and $\mathcal{G}_{J}(\delta,M) \subseteq \mathcal{G}_{J}(\delta_1,M_1)$ whenever
$\delta \ge \delta_1\ge 0$ and $M_1\ge M\ge 0$.

\begin{pro}\label{P5}
	Let $M,\delta\ge 0$. Then $g\in 	\mathcal{G}_{J}(\delta,M)$ if and only if $\psi^{-1}\circ g\circ\psi  \in 	\mathcal{F}_{I}(\delta,M)$, where $\psi (x)=e^x$ and $I=\log(J):=\{\log x:x\in J\}$.
\end{pro}
\begin{proof}
Let	$g\in 	\mathcal{G}_{J}(\delta,M)$ and $I:=\log (J)=[a,b]$. Then $a=\log c$ and $b=\log d$.  Clearly, $f:=\psi^{-1}\circ g\circ\psi  \in \mathcal{C}_b(\mathbb{R})$. Also, we have $f(a)=\log g(e^a)=\log g(c)=\log c=a$, and similarly $f(b)=b$. So, $I\subseteq \mathcal{R}(f)$. The reverse inclusion follows by definitions of $f$ and $I$, because $\mathcal{R}(g)=J$. Therefore $\mathcal{R}(f)=I$.

 Next, let $x, y\in I$ with $x\ge y$. Then there exist $u,v \in J$ with $u\ge v$ such that $x=\log u$ and $y=\log v$. So, from the assumption on $g$, we have
\begin{eqnarray*}
\left(\frac{u}{v}\right)^\delta \le \frac{g(u)}{g(v)}\le \left(\frac{u}{v}\right)^M,
\end{eqnarray*}
implying that
\begin{eqnarray*}
\delta	\log \left(\frac{e^x}{e^y}\right) \le \log \left(\frac{g(e^x)}{g(e^y)}\right)\le M\log \left(\frac{e^x}{e^y}\right).
\end{eqnarray*}
i.e., $\delta(x-y)\le f(x)-f(y)\le M(x-y)$. Therefore $f \in 	\mathcal{F}_{I}(\delta,M)$. The converse follows similarly.
\end{proof}

\begin{pro}\label{P1}
	If $M<1$ or $\delta >1$, then $\mathcal{G}_{J}(\delta,M)=\emptyset$.
If $M=1$ or $\delta=1$, then $\mathcal{G}_{J}(\delta,M):=\{g\in \mathcal{C}_b(\mathbb{R}_+): g|_{J}={\rm id}\}$.
\end{pro}

\begin{proof}
Let $g\in \mathcal{G}_{J}(\delta,M)$.	Then by Proposition \ref{P5},  $f:=\psi^{-1}\circ g\circ\psi  \in 	\mathcal{F}_{I}(\delta,M)$, where $\psi (x)=e^x$ and $I=\log(J)$. So, for any $x, y\in I$ such that $x\ge y$, we have
	\begin{eqnarray}\label{5}
	\delta(x-y)\le f(x)-f(y)\le M(x-y).
	\end{eqnarray}
	If $M<1$, then by setting $y=a$ in \eqref{5}, we get that $f(x)<x, \forall x\in I$ with $x>a$. This is a contradiction to the fact that $f(b)=b$, because $b>a$. So, $\mathcal{F}_{I}(\delta,M)=\emptyset$ and hence $\mathcal{G}_{J}(\delta,M)=\emptyset$, whenever $M<1$. A similar argument holds when $\delta>1$.
	
	If $M=1$, then from \eqref{5}, we have
		\begin{eqnarray}\label{6}
 f(x)-f(y)\le (x-y), \quad \forall x,y \in I~\text{with}~x\ge y.
	\end{eqnarray}
	Now for $x=b$, \eqref{6} implies that $f(y)\ge y, \forall y\in I$ with $y<b$. Moreover, setting $y=a$ in \eqref{6}, we have $f(x)\le x, \forall x\in I$ with $x> a$. Thus $f(x)=x, \forall x\in I$, and therefore $f|_I={\rm id}$. This implies that $g|_J={\rm id}$.  The reverse inclusion is trivial. So, $\mathcal{G}_{J}(\delta,M):=\{g\in \mathcal{C}_b(\mathbb{R}_+): g|_{J}={\rm id}\}$. A similar argument holds when $\delta=1$.
\end{proof}

\begin{pro}\label{Fdm complete}
The set $\mathcal{F}_{I}(\delta,M)$
is a complete metric space  under the metric induced by $\|\cdot \|$.
\end{pro}

\begin{proof}
	It can be easily seen that $\mathcal{F}_{I}(\delta,M)$ is a closed subset of $\mathcal{C}_b(\mathbb{R})$. So, since  $\mathcal{C}_b(\mathbb{R})$ is complete with respect to the metric induced by $\|\cdot \|$, it follows that  $\mathcal{F}_{I}(\delta,M)$ is also complete.
\end{proof}


In view of Proposition \ref{P1}, we cannot seek solutions of
\eqref{1} without imposing conditions on $M$ and $\delta$. So, henceforth we assume that $0<\delta\le 1\le M$.
We need the following
six technical lemmas,
last three of which
look similar to some lemmas given in \cite{Mu-Su-2009} but
we have to rewrite their proofs carefully because of the following difference:
It is assumed in \cite{Mu-Su-2009} that $f\in \mathcal{C}(I,I)$ and $f$ is a homeomorphism of $I$ onto itself,
implying that $f^{-1}$ is well defined on the whole domain $I$ of $f$,
but our this paper deals with $f\in \mathcal{C}(\mathbb{R}, I)$ satisfying that $f|_I$ is a homeomorphism of $I$ onto itself.
So, $f$ is not a homeomorphism on $\mathbb{R}$,
that is, $f^{-1}$ is not defined on the whole $\mathbb{R}$.
In this case,
we can consider only the inverse of $f|_{I}$. For a specific instance,
the conclusion $L_f\in \mathcal{F}_{I}(K_0,K_1)$, made in Lemma 3.2 of \cite{Mu-Su-2009}, is not true here, simply because  we have defined $L_f$ only on $I$. However, even if we define it on the whole of $\mathbb{R}$, it does not belong to  $\mathcal{F}_{I}(K_0,K_1)$, because in that case $\mathcal{R}(L_f)\ne I$. So, in view of this, we include their proofs here in order to avoid confusions.

\begin{lm}{\rm (\cite{zhang1990})} \label{L2}
	Let $f, g\in \mathcal{C}(I,I)$ satisfy $|f(x)-f(y)|\le M|x-y|$ and $|g(x)-g(y)|\le M|x-y|$ for all $x, y \in I$, where $M\ge 1$. Then
	\begin{eqnarray}
	\|f^k-g^k\|_I\le \left(\sum_{j=0}^{k-1}M^j\right)\|f-g\|_I~~\text{for}~k=1,2,\ldots.
	\end{eqnarray}
\end{lm}

\begin{lm}{\rm (\cite{Zhang-Baker})} \label{L1}
Let $f\in \mathcal{C}(I, I)$ satisfy
$f(a)=a$, $f(b)=b$ and $\delta(x-y)\le f(x)-f(y)\le M(x-y)$ for all $x, y \in I$ with $x\ge y$, where $0<\delta\le 1\le M$.
 Then $f$ is a homeomorphism of $I$ onto itself and
	\begin{eqnarray}\label{7}
		\frac{1}{M}(x-y)\le f^{-1}(x)-f^{-1}(y)\le \frac{1}{\delta}(x-y),
	\end{eqnarray}
for each $x,y \in I~\text{with}~ x\ge y$.
\end{lm}


\begin{lm}{\rm (\cite{Zhang-Baker})}\label{L3}
	Let $f,g:I\to I$ be homeomorphisms such that
	$\delta(x-y)\le f(x)-f(y)\le M(x-y)$  and $\delta(x-y)\le g(x)-g(y)\le M(x-y)$ for all $x, y \in I$ with $x\ge y$,
 where $0<\delta\le 1\le M$. Then $$\delta\|f^{-1}-g^{-1}\|_I\le \|f-g\|_I\le M\|f^{-1}-g^{-1}\|_I.$$
\end{lm}

For $\alpha_k \ge 0$ ($1\le k\le n$) with $\sum_{k=1}^{n}\alpha_k=1$ and $f\in  \mathcal{F}_{I}(\delta,M)$, define $L_f: I\to I$ by
$$L_f(x)=\alpha_1 x+\alpha_2 f(x)+\cdots+ \alpha_nf^{n-1}(x),\quad x\in I.$$

\begin{lm}\label{Lf estimate}
	Let	$f\in  \mathcal{F}_{I}(\delta,M)$, where $0<\delta\le 1\le M$. Then $L_f(a)=a$, $L_f(b)=b$, $\mathcal{R}(L_f)=I$ and
	\begin{eqnarray}\label{8}
	K_0(x-y)\le L_f(x)-L_f(y)\le K_1(x-y),
	\end{eqnarray}
for each $x,y \in I~\text{with}~ x\ge y$, where
 $K_0:=\sum_{k=1}^{n}\alpha_k\delta^{k-1}$ and $K_1:=\sum_{k=1}^{n}\alpha_kM^{k-1}$.
\end{lm}

\begin{proof}
It can be easily seen that $L_f(a)=a$, $L_f(b)=b$ and $\mathcal{R}(L_f)=I$. Also, for any $x,y\in I$ with $x\ge y$, we have
\begin{eqnarray*}
L_f(x)-L_f(y)&=&\sum_{k=1}^{n}\alpha_kf^{k-1}(x)-\sum_{k=1}^{n}\alpha_kf^{k-1}(y)\\
&=&\sum_{k=1}^{n}\alpha_k(f^{k-1}(x)-f^{k-1}(y))\\
&\le &\left(\sum_{k=1}^{n}\alpha_kM^{k-1}\right)(x-y)=K_1(x-y)
\end{eqnarray*}
and
\begin{eqnarray*}
	L_f(x)-L_f(y)&=&\sum_{k=1}^{n}\alpha_kf^{k-1}(x)-\sum_{k=1}^{n}\alpha_kf^{k-1}(y)\\
	&=&\sum_{k=1}^{n}\alpha_k(f^{k-1}(x)-f^{k-1}(y))\\
	&\ge &\left(\sum_{k=1}^{n}\alpha_k\delta^{k-1}\right)(x-y)=K_0(x-y).
\end{eqnarray*}
This completes the proof.
\end{proof}

\begin{lm}\label{Lf inverse lipschits}
	Let $0<\delta \le 1\le  M$ and	$f\in  \mathcal{F}_{I}(\delta,M)$. Then
	\begin{eqnarray}\label{9}
	\frac{1}{K_1}(x-y)\le L_f^{-1}(x)-L_f^{-1}(y)\le \frac{1}{K_0}(x-y),
	\end{eqnarray}
for each $x,y \in I~\text{with}~ x\ge y$, where $K_0, K_1$ are as in Lemma \ref{Lf estimate}.
\end{lm}

\begin{proof}
Follows from the proof of Lemma \ref{L1}, by noting
from Lemma \ref{Lf estimate} that  $L_f(a)=a$, $L_f(b)=b$, $\mathcal{R}(L_f)=I$ and $L_f$ satisfies \eqref{8} with $0<K_0\le 1\le K_1$.
\end{proof}

\begin{lm}\label{Lf inverse estimate}
Let $0<\delta \le 1\le  M$ and	$f_1, f_2\in  \mathcal{F}_{I}(\delta,M)$. Then $$\|L_{f_1}-L_{f_2}\|_I\le K_2\|f_1-f_2\|_I~\text{and}~\|L_{f_1}^{-1}-L_{f_2}^{-1}\|_I\le \frac{K_2}{K_0}\|f_1-f_2\|_I,$$
	where $K_0, K_1$ are as in Lemma \ref{Lf estimate} and
	$K_2:= \sum_{k=2}^{n}\alpha_k(\sum_{j=0}^{k-2}M^j)$.
	\end{lm}

\begin{proof}
	Let 	$f_1, f_2\in  \mathcal{F}_{I}(\delta,M)$. Then for each $x\in I$, we have
	\begin{eqnarray*}
	|L_{f_1}(x)-L_{f_2}(x)|&=&\left| \sum_{k=1}^{n}\alpha_kf_1^{k-1}(x)-\sum_{k=1}^{n}\alpha_kf_2^{k-1}(x)\right|\\
	&\le & \sum_{k=2}^{n}\alpha_k|f_1^{k-1}(x)-f_2^{k-1}(x)|\\
	&\le & \sum_{k=2}^{n}\alpha_k\|f_1^{k-1}-f_2^{k-1}\|_I\\
	&\le& \left( \sum_{k=2}^{n}\alpha_k\left(\sum_{j=0}^{k-2}M^j\right)\right)\|f_1-f_2\|_I~~(\text{using Lemma \ref{L2}})\\
	&=& K_2\|f_1-f_2\|_I,
	\end{eqnarray*}
implying that
\begin{eqnarray}\label{10}
\|L_{f_1}-L_{f_2}\|_I\le K_2\|f_1-f_2\|_I.
\end{eqnarray}
Moreover, since $L_{f_1}, L_{f_2}:I\to I$ are homeomorphisms satisfying \eqref{8} with $0<K_0\le 1 \le  K_1$, by Lemma \ref{L3} we have
\begin{eqnarray}\label{11}
\|L_{f_1}^{-1}-L_{f_2}^{-1}\|_I\le \frac{1}{K_0}\|L_{f_1}-L_{f_2}\|_I.
\end{eqnarray}
Therefore \eqref{10} and \eqref{11} together implies that
\begin{eqnarray}\label{13}
\|L_{f_1}^{-1}-L_{f_2}^{-1}\|_I\le \frac{K_2}{K_0}\|f_1-f_2\|_I.
\end{eqnarray}
This completes the proof.
\end{proof}


\section{Existence, uniqueness and stability}\label{fp method}

In this section we give results on existence, uniqueness and stability.

\begin{thm}\label{existance}
	Let $0<\alpha_1<1$, $\alpha_k \ge 0$ for $2\le k\le n$
	 such that $\sum_{k=1}^{n}\alpha_k=1$
and	$G\in \mathcal{G}_J(K_1\delta, K_0M)$, where $J=[c,d]$, $c<d$ and $0<\delta\le 1\le M$.
 Let $K_0, K_1$ and $K_2$ be as defined  in Lemmas \ref{Lf estimate} and \ref{Lf inverse estimate}, respectively.
 If $K_2<K_0$, then \eqref{1}
	has a unique solution in $\mathcal{G}_J(\delta, M)$.
\end{thm}

\begin{proof}
Let $G\in \mathcal{G}_J(K_1\delta, K_0M)$, $a:=\log c$ and $b:=\log d$. Then we obtain the interval $I=[a,b]$ with $a<b$, which satisfies $I=\log J$.
By Proposition \ref{P5}, we have $F:=\psi^{-1}\circ G\circ \psi \in  \mathcal{F}_I(K_1\delta, K_0M)$, where  $\psi (x)=e^x$.
	
Define $T:\mathcal{F}_I(\delta, M)\to \mathcal{C}_b(\mathbb{R})$ by
$$
Tf(x)=L_f^{-1}(F(x)),\quad x\in \mathbb{R}.
$$
By definitions of $F$ and $L_f$, we have $Tf(a)=a$ and $Tf(b)=b$. This implies that $I\subseteq \mathcal{R}(Tf)$. Also, since $L_f^{-1}:I\to I$, we have $\mathcal{R}(L_f^{-1})\subseteq I$, and therefore $\mathcal{R}(Tf)\subseteq I$. So, $\mathcal{R}(Tf)=I$. Further, for any $x,y\in I$ with $x\ge y$, as $F\in \mathcal{F}_I(K_1\delta, K_0M)$, we have
	\begin{eqnarray*}
		Tf(x)-Tf(y)&=&L_f^{-1}(F(x))-L_f^{-1}(F(y))\\
		&\le &\frac{1}{K_0}(F(x)-F(y))~~~~(\text{by using Lemma}~\ref{Lf inverse lipschits})\\
		&\le &\frac{1}{K_0}K_0M(x-y)\\
		&=&M(x-y)
	\end{eqnarray*}
	and
	\begin{eqnarray*}
		Tf(x)-Tf(y)&=&L_f^{-1}(F(x))-L_f^{-1}(F(y))\\
		&\ge &\frac{1}{K_1}(F(x)-F(y))~~~~(\text{again by using Lemma}~\ref{Lf inverse lipschits})\\
		&\ge &\frac{1}{K_1}K_1\delta(x-y)\\
		&=&\delta(x-y).
	\end{eqnarray*}
	Hence $Tf\in \mathcal{F}_I(\delta, M)$, which proves that $T$ is a self-map on $\mathcal{F}_I(\delta, M)$.
	
	We now prove that $T$ is a contraction. For $f_1, f_2 \in \mathcal{F}_I(\delta, M)$ and $x\in \mathbb{R}$, we have
	\begin{eqnarray}\label{4}
	|Tf_1(x)-Tf_2(x)|&=&|L_{f_1}^{-1}(F(x))-L_{f_2}^{-1}(F(x))|\nonumber\\
	&\le&\|L_{f_1}^{-1}-L_{f_2}^{-1}\|_I~~(\text{since}~F(x)\in I)\nonumber\\
	&\le &\frac{K_2}{K_0}\|f_1-f_2\|_I~~~(\text{by using Lemma}~\ref{Lf inverse estimate})\nonumber\\
	&\le &\frac{K_2}{K_0}\|f_1-f_2\|,
	\end{eqnarray}
	implying that $\|Tf_1-Tf_2\|\le \frac{K_2}{K_0}\|f_1-f_2\|$. Since $0<K_2<K_0$, it follows that $T$ is a contraction. By Proposition \ref{Fdm complete}, $\mathcal{F}_I(\delta, M)$ is complete, and hence by Banach's contraction principle, $T$ has a unique fixed point in $\mathcal{F}_I(\delta, M)$.  That is, there exists unique $f\in \mathcal{F}_I(\delta, M)$ such that $L_f^{-1}(F(x))=f(x), \forall x\in \mathbb{R}$, which proves that $f$ is the unique solution of \eqref{(*)} in $\mathcal{F}_I(\delta, M)$. This implies by Propositions \ref{P3} and \ref{P5} that $g:=\psi \circ f\circ\psi^{-1}$ is the unique solution of \eqref{1} in $\mathcal{G}_J(\delta, M)$.
The proof is completed.
\end{proof}

Although Lemmas \ref{Lf estimate}, \ref{Lf inverse lipschits} and \ref{Lf inverse estimate} are true for $\alpha_1\in [0,1]$, in Theorem \ref{existance} we have assumed that $\alpha_1\in (0,1)$ for the following reason: If $\alpha_1=1$, then $g=G$ is the unique solution of \eqref{1} so that the problem is trivial. On the other hand, if $\alpha_1=0$, then we have $K_0=\sum_{k=2}^{n}\alpha_k\delta^{k-1}$. So, the condition (in Theorem \ref{existance}) $K_2<K_0$ is not satisfied, because $$K_2= \sum_{k=2}^{n}\alpha_k\left(\sum_{j=0}^{k-2}M^j\right)>\sum_{k=2}^{n}\alpha_k\ge \sum_{k=2}^{n}\alpha_k\delta^{k-1}=K_0.$$
 Thus this theorem is not true for $\alpha_1=0$. In particular,
 one cannot solve the iterative root problem $g^n=G$ on $\mathbb{R}_+$ using this theorem.

\begin{coro}\label{C1}
	In addition to the assumptions of Theorem \ref{existance}, suppose that $G|_J={\rm id}$. If $K_2<K_0$, then $G$ is the unique solution of \eqref{1} in  $\mathcal{G}_J(\delta, M)$.
\end{coro}
\begin{proof}
	Follows by Theorem \ref{existance}, because clearly $G$ is a solution of \eqref{1} in $\mathcal{G}_J(\delta, M)$.
\end{proof}


Under the assumptions of Theorem \ref{existance}, we will show that the solution obtained depends continuously on the function $G$.
More precisely, we have the following.

\begin{thm}\label{stability}
	In addition to assumptions of Theorem \ref{existance}, suppose that 	$G_1\in \mathcal{G}_J(K_1\delta, K_0M)$ and $g_1 \in \mathcal{G}_J(\delta, M)$ satisfy that
$\prod_{k=1}^{n}(g_1^{k}(x))^{\alpha_k}=G_1(x)$ for all $x\in \mathbb{R}_+$.
	Then
	\begin{eqnarray}\label{18}
\|g-g_1\|\le \frac{d}{c(K_0-K_2)}\|G-G_1\|.
	\end{eqnarray}
\end{thm}

\begin{proof}
	Given $G, G_1, g$ and $g_1$ as above, let $F(x)=\log G(e^x), F_1(x)=\log G_1(e^x), f(x)=\log g(e^x)$ and $f_1(x)=\log g_1(e^x), \forall x\in \mathbb{R}$.
	Since $G, G_1 \in \mathcal{G}_J(K_1\delta, K_0M)$, by Proposition \ref{P5}, we have $F, F_1\in \mathcal{F}_I(K_1\delta, K_0M)$, where $I=[a,b]$ such that $a=\log c$ and $b=\log d$.
Using a similar argument, we see that $f, f_1\in \mathcal{F}_I(\delta, M)$.
Moreover, $f$ and $f_1$ satisfy equation \eqref{(*)} and the equation
\begin{eqnarray*}\label{12}
		\sum_{k=1}^{n}\alpha_k f_1^k(x)=F_1(x), \quad x\in \mathbb{R},
	 \end{eqnarray*}
	 respectively, implying that $L_f^{-1} (F(x))=f(x)$ and $L_{f_1}^{-1} (F_1(x))=f_1(x), \forall x\in \mathbb{R}$.
	Therefore, for each $x\in \mathbb{R}$,
	\begin{eqnarray*}
	|f(x)-f_1(x)|&=& |L_f^{-1} (F(x))-L_{f_1}^{-1} (F_1(x))|\\
	&\le & |L_f^{-1} (F(x))-L_{f_1}^{-1} (F(x))|+|L_{f_1}^{-1} (F(x))-L_{f_1}^{-1} (F_1(x))|\\
	&\le & \|L_f^{-1}-L_{f_1}^{-1}\|_I+\frac{1}{K_0} |F(x)-F_1(x)|~~(\text{using}~\eqref{9})\\
	&\le &  \frac{K_2}{K_0}\|f-f_1\|_I+ \frac{1}{K_0}\|F-F_1\|_I~~(\text{using}~\eqref{13})\\
	&\le &  \frac{K_2}{K_0}\|f-f_1\|+ \frac{1}{K_0}\|F-F_1\|,
	\end{eqnarray*}
and hence
\begin{eqnarray*}
\|f-f_1\|	\le   \frac{K_2}{K_0}\|f-f_1\|+ \frac{1}{K_0}\|F-F_1\|.
\end{eqnarray*}
Since $K_2<K_0$, the above inequality shows
	\begin{eqnarray}\label{14}
\|f-f_1\|\le \frac{1}{K_0-K_2}\|F-F_1\|.
\end{eqnarray}
Since the map $x\mapsto e^x$ is continuously differentiable on $I$ with bounded derivative, it is a Lipschitzian
map on $I$. In fact,
\begin{eqnarray*}
|e^x-e^y|<e^b|x-y|, \quad \forall x, y\in I.
\end{eqnarray*}
So, for each $x\in \mathbb{R}_+$, we have
\begin{eqnarray*}
|g(x)-g_1(x)|&=&|e^{f(\log x)}-e^{f_1(\log x)}|\\
&<&e^b|f(\log x)-f_1(\log x)|\\
&\le & e^b\|f-f_1\|,
\end{eqnarray*}
implying that
\begin{eqnarray}\label{17}
\|g-g_1\|&\le& e^b \|f-f_1\|\nonumber\\
&\le & \frac{d}{K_0-K_2} \|F-F_1\|~~(\text{using} ~\eqref{14}).
\end{eqnarray}
Since the map $x\mapsto \log x$ is continuously differentiable on $J$ with bounded derivative, it is a Lipschitzian
map on $J$. In fact,
\begin{eqnarray*}
|\log x-\log y|<\frac{1}{c}|x-y|, \quad \forall x, y\in J.
\end{eqnarray*}
Therefore, for each $x\in \mathbb{R}$, we have
\begin{eqnarray*}
|F(x)-F_1(x)|&=&|\log G(e^x)-\log G_1(e^x)|\\
&\le &\frac{1}{c}|G(e^x)-G_1(e^x)|\\
&\le &\frac{1}{c}\|G-G_1\|,
\end{eqnarray*}
implying that
\begin{eqnarray}\label{22}
\|F-F_1\|\le \frac{1}{c} \|G-G_1\|.
\end{eqnarray}
Then \eqref{18} follows from \eqref{17} using \eqref{22}.
\end{proof}

The assumptions that $0< \alpha_1 <1$ and $\sum_{k=1}^{n}\alpha_k=1$ made in Theorem \ref{existance} are not strong.
In fact, if $\alpha_1>1$
 or $\sum_{k=1}^{n}\alpha_k>1$, then we can divide
 all the exponents $\alpha_k$s in \eqref{1} by $\sum_{k=1}^{n}\alpha_k$ to get the normalized equation,
but the assumptions on $G$ have to be changed suitably.
Moreover, by using the above observation, Theorem \ref{existance} and Proposition \ref{P4}, we can indeed extend the solutions of \eqref{1} on $\mathbb{R}_+$ to $\mathbb{R}_-$ whenever $\alpha_k\in \mathbb{Z}$ for $1\le k\le n$ such that $\sum_{k=1}^{n}\alpha_k$ is odd.


\section{Construction of solutions}\label{construction}

The method used in section 3 is an application of Banach's fixed point theorem, which gives an recursive algorithm
to approach the unique solution.
Unlike section 3, in this section we can use another method, sewing piece by piece, to
find more continuous solutions of \eqref{1} on the whole $\mathbb{R}_+$.

Consider \eqref{1} with real $\alpha_k$'s, $1\le k\le n$, and without loss of generality assume that $\alpha_n\ne 0$.
 Then \eqref{1} and its modified equation \eqref{(*)}
can be represented equivalently as
\begin{eqnarray}\label{lcp for g}
g^n(x)=\prod_{k=1}^{n-1}(g^k(x))^{\lambda_k} G(x)
\end{eqnarray}
and
\begin{eqnarray}\label{lcp}
f^n(x)=\sum_{k=1}^{n-1}\lambda_k f^k(x)+F(x)
\end{eqnarray}
respectively, where $\lambda_k$'s are real for $1 \le k\le n-1$. Let $\lambda:=\sum_{k=1}^{n-1}\lambda_k$. We will discuss for
$\lambda \ge 0$ and $\lambda < 0$ separately.

First, we consider the case that $\lambda \ge 0$.
In 2007, Xu and Zhang  \cite{Xu-Zhang} proved the existence of continuous solution of \eqref{lcp} on the compact interval $I$ with the assumption that $\lambda\in [0,1)$. In what follows, solving \eqref{lcp} with $\lambda\in [0,1)$ on the whole $\mathbb{R}$, we obtain solutions of
\eqref{lcp for g} on $\mathbb{R}_+$.

Let $\mathbb{I}:=|a,b|$, where $|a,b|$ denotes either an open interval $(a,b)$, a semi-closed interval $[a,b)$ or $(a,b]$, or a closed interval $[a,b]$ in $\mathbb{R}$, and one or both of the endpoints of $\mathbb{I}$ may be infinite.
Let $\mathbb{J}=|c,d|$ be an interval in $\mathbb{R}_+$, where $c$ and $d$ may be $0$ and $\infty$ respectively.
For $\zeta \in \bar{\mathbb{I}}$, the closure of $\mathbb{I}$,
$\eta \in  \bar{\mathbb{J}}$ and $\lambda\in [0,1)$, let
\begin{eqnarray*}
	R_{\zeta, \lambda}[\mathbb{R}; \mathbb{I}]
	\!\!\!\!&:=&\!\!\!\! \{f\in \mathcal{C}_b(\mathbb{R}): f|_\mathbb{I}~\mbox{is strictly increasing and satisfies {\bf (A1)} and {\bf (A2)}
	}
	\},
	\\
	S_{\eta, \lambda}[\mathbb{R}_+; \mathbb{J}]
	\!\!\!\!&:=&\!\!\!\! \{g\in \mathcal{C}_b(\mathbb{R}_+): g|_\mathbb{J}~\mbox{is strictly increasing and satisfies {\bf (B1)} and {\bf (B2)}
	}
	\},
\end{eqnarray*}
where
\begin{enumerate}
	\item [{\bf (A1)}] \quad $(f(x)-(1-\lambda)x)(\zeta-x)>0$ for $x\ne \zeta$,
	\item [{\bf (A2)}] \quad $(f(x)-(1-\lambda)\zeta)(\zeta-x)<0$ for $x\ne \zeta$,
	\item [{\bf (B1)}] \quad $(g(x)-x^{1-\lambda})(\eta-x)>0$ for $x\ne \eta$,
	\item [{\bf (B2)}] \quad $(g(x)-\eta^{1-\lambda})(\eta-x)<0$ for $x\ne \eta$.
\end{enumerate}

\begin{pro}\label{P6}
	Let $\lambda\in [0,1)$. Then a map $g\in 	S_{\eta, \lambda}[\mathbb{R}_+; \mathbb{J}]$ for $\eta \in \bar{\mathbb{J}}$ if and only if $f=\psi^{-1} \circ g\circ \psi \in 	R_{\zeta, \lambda}[\mathbb{R}; \mathbb{I}]$, where $\psi (x)=e^x$, $\zeta=\log \eta$ and $\mathbb{I}=\log (\mathbb{J})$.
\end{pro}

\begin{proof}
	Let $g\in 	S_{\eta, \lambda}[\mathbb{R}_+; \mathbb{J}]$, where $\eta \in \bar{\mathbb{J}}$. Since $f|_\mathbb{I}=\psi^{-1}\circ (g|_\mathbb{J})\circ \psi$ and $g|_\mathbb{J}$ is strictly increasing, clearly $f|_\mathbb{I}$ is also strictly increasing. In order to prove that $f$ satisfies condition {\bf (A1)}, consider any $x\in \mathbb{I}$ such that $x\ne \zeta$. Let $y\in \mathbb{J}$ be such that $x=\log y$. Then either $g(y)< y^{1-\lambda}$ or  $g(y)> y^{1-\lambda}$ according as either $y>\eta$ or $y<\eta$, respectively. This implies that either $f(x)<(1-\lambda)x$ or $f(x)>(1-\lambda)x$ according as ether $x>\zeta$ or $x<\zeta$, respectively.  In any case, we have $(f(x)-(1-\lambda)x)(\zeta-x)>0$. By a similar argument, using condition {\bf (B2)} for $g$, we can prove that $f$ satisfies condition {\bf (A2)}.  Hence $f\in 	R_{\zeta, \lambda}[\mathbb{R}; \mathbb{I}]$.  The converse follows similarly.
\end{proof}

\begin{lm}{\rm (\cite{Xu-Zhang})} \label{L6}
	Let $\lambda\in [0,1)$ and $F\in R_{a, \lambda}[\mathbb{I}; \mathbb{I}]$.
Then for arbitrary $x_0\in (a,b|$, \eqref{lcp} has a solution in $R_{a, 0}[I_1; I_1]$, where $I_1=[a,x_0]$. More concretely, for every arbitrary $x_0\in (a, b|$, there exists a strictly decreasing sequence $(x_1, x_2, \ldots, x_{n-1})$ in $(a, x_0)$ such that the
	sequence $(x_m)$ defined recursively by
	\begin{eqnarray}\label{x_n}
	x_{n+m}=\sum_{j=1}^{n-1}\lambda_j x_{j+m}+F(x_{m})\quad \text{for}~m\ge 0
	\end{eqnarray}
	satisfies the conditions {\rm \bf (i)} $x_{m+1}\in  (a, x_m)$ for $m\ge 1$, { \rm \bf (ii)} $(a,x_0]=\bigcup\limits_{m=1}^{\infty}[x_m, x_{m-1}]$,
	and
	\begin{eqnarray*}
		f(x):=\left\{\begin{array}{cll}
			a&\text{if}&x=a,\\
			f_m(x)&\text{if}&x\in [x_m, x_{m-1}], ~m\ge 1
		\end{array}\right.
	\end{eqnarray*}
	is a solution of \eqref{lcp} in $R_{a, 0}[I_1; I_1]$, where $f_j:[x_j, x_{j-1}]\to [x_{j+1}, x_j]$ is an arbitrary order-preserving homeomorphism for $1\le j\le n-1$ and $f_m: [x_m, x_{m-1}]\to [x_{m+1}, x_m]$ is the order-preserving homeomorphism defined recursively by
	\begin{eqnarray*}
		f_m(x)&=&\lambda_{n-1} x+\lambda_{n-2}f_{m-1}^{-1}(x)+\cdots+\lambda_1f_{m-n+2}^{-1}\circ f_{m-n+3}^{-1}\circ \cdots \circ f_{m-1}^{-1}(x)\\
		& &+ F\circ f_{m-n+1}^{-1}\circ f_{m-n+2}^{-1}\circ \cdots \circ f_{m-1}^{-1}(x), \quad  x\in [x_m, x_{m-1}], ~~\mbox{ for } m\ge n.
	\end{eqnarray*}
\end{lm}

\begin{lm}{\rm (\cite{Xu-Zhang})} \label{L7}
	Let $\lambda\in [0,1)$ and $F\in R_{b, \lambda}[\mathbb{I}; \mathbb{I}]$. Then for every arbitrary $x_0\in |a,b)$, \eqref{lcp} has a solution in $R_{b, 0}[I_2; I_2]$, where $I_2=[x_0,b]$. More concretely, for every arbitrary $x_0\in |a, b)$, there exists a strictly increasing sequence $(x_1, x_2, \ldots, x_{n-1})$ in $(x_0, b)$ such that the
	sequence $(x_m)$ defined recursively by \eqref{x_n} satisfies the conditions {\rm \bf(i)} $x_{m+1}\in (x_m,b)$ for $m\ge 1$, {\rm \bf (ii)} $[x_0,b)=\bigcup\limits_{m=1}^{\infty}[x_{m-1}, x_m]$,  and
	\begin{eqnarray*}
		f(x):=\left\{\begin{array}{cll}
			f_m(x)&\text{if}&x\in [x_{m-1}, x_m], ~m\ge 1,
\\
			b&\text{if}&x=b,\\
		\end{array}\right.
	\end{eqnarray*}
	is a solution of \eqref{lcp} in $R_{b, 0}[I_2; I_2]$, where $f_j:[x_{j-1}, x_j]\to [x_j, x_{j+1}]$ is an arbitrary order-preserving homeomorphism for $1\le j \le n-1$ and $f_m: [x_{m-1}, x_m]\to [x_m, x_{m+1}]$ is the order-preserving homeomorphism defined recursively by
		\begin{eqnarray*}
		f_m(x)&=&\lambda_{n-1} x+\lambda_{n-2}f_{m-1}^{-1}(x)+\cdots+\lambda_1f_{m-n+2}^{-1}\circ f_{m-n+3}^{-1}\circ \cdots \circ f_{m-1}^{-1}(x)\\
		& &+ F\circ f_{m-n+1}^{-1}\circ f_{m-n+2}^{-1}\circ \cdots \circ f_{m-1}^{-1}(x), \quad x\in [x_{m-1}, x_m], ~~\mbox{ for } m\ge n.
	\end{eqnarray*}
\end{lm}

\begin{theorem}\label{T3}
	Let $\lambda \in [0,1)$ and $G\in S_{c, \lambda}[\mathbb{R}_+; \mathbb{J}]$ such that $\mathcal{R}(G)=\mathcal{R}(G|_\mathbb{J})$,  where $\mathbb{J}=|c,d]$.
Then \eqref{lcp for g} has solutions
in $S_{c, 0}[\mathbb{R}_+; \mathbb{J}]$. Moreover, each solution depends on $n-1$ arbitrarily chosen orientation-preserving homeomorphisms $f_j:[x_j, x_{j-1}]\to [x_{j+1},x_j]$, $j=1,2,\ldots, n-1$,  where $x_0=b$ and $x_1, x_2, \ldots, x_n$ are given as in Lemma \ref{L6}.
\end{theorem}
\begin{proof}
	Given $G\in S_{c, \lambda}[\mathbb{R}_+; \mathbb{J}]$, by Proposition \ref{P6}, we have $F=\psi^{-1} \circ G\circ \psi \in 	R_{a, \lambda}[\mathbb{R}; \mathbb{I}]$, where $\psi (x)=e^x$ and $\mathbb{I}=\log (\mathbb{J})$. Also, since $\mathcal{R}(G)=\mathcal{R}(G|_\mathbb{J})$, it follows that $\mathcal{R}(F)=\mathcal{R}(F|_\mathbb{I})$.
	So, $F_1:=F|_\mathbb{I} \in R_{a, \lambda}[\mathbb{I}; \mathbb{I}]$. Therefore by Lemma \ref{L6}, \eqref{lcp} has a solution $\phi_1$ in $R_{a, 0}[\mathbb{I}; \mathbb{I}]$. Let $f$ be the extension of map $\phi_1$ to $\mathbb{R}$ defined by
	\begin{eqnarray}\label{extension}
	f(x)=\phi_1\circ F_1^{-1}\circ F(x), \quad x\in \mathbb{R}.
	\end{eqnarray}
We assert that $f$ is a solution of \eqref{lcp} in $R_{a, 0}[\mathbb{R}; \mathbb{I}]$.
Being a strictly increasing continuous map,  $F_1:\mathbb{I} \to \mathcal{R}(F_1)$ has the inverse $F_1^{-1}$, which is also strictly increasing and continuous on  $\mathcal{R}(F_1)$. Therefore,
as $\mathcal{R}(F)=\mathcal{R}(F_1)$, clearly $f$ is a well-defined map on $\mathbb{R}$. Also, $f$ is continuous on $\mathbb{R}$, being the composition of continuous maps $\phi_1, F_1^{-1}$ and $F$. Further,
	 as $\phi_1 \in R_{a, 0}[\mathbb{I}; \mathbb{I}]$, it follows that $f$ is strictly increasing on $\mathbb{I}$, and satisfies the conditions {\bf (A1)} and {\bf (A2)}. Therefore $f\in R_{a, 0}[\mathbb{R}; \mathbb{I}]$. Moreover, for each $x\in \mathbb{R}$,
	\begin{eqnarray*}
	f^n(x)-\sum_{k=1}^{n-1}\lambda_k f^k(x)&=& f^{n-1}(f(x))-\sum_{k=1}^{n-1}\lambda_k f^{k-1}(f(x))\\
	&=&f^{n-1}|_I(f(x))-\sum_{k=1}^{n-1}\lambda_k f^{k-1}|_I(f(x))\\
	& &~~~~~~~~~~~~~~~~~~~~~~~~(\text{since}~ f(x)\in \mathcal{R}(f)=\mathcal{R}(f|_\mathbb{I})\subseteq \mathbb{I})
\end{eqnarray*}
\begin{eqnarray*}
	&=& (f|_I)^{n-1} (f(x))-\sum_{k=1}^{n-1}\lambda_k (f|_I)^{k-1}(f(x))\\ 	
	&=& \phi_1^{n-1}(\phi_1\circ F_1^{-1}\circ F(x))-\sum_{k=1}^{n-1}\lambda_k \phi_1^{k-1}(\phi_1\circ F_1^{-1}\circ F(x))\\
	&=& \phi_1^n(F_1^{-1}\circ F(x))-\sum_{k=1}^{n-1}\lambda_k \phi_1^k(F_1^{-1}\circ F(x))\\
	&=&F_1(F_1^{-1}\circ F(x))~~~~~(\text{since}~ F_1^{-1}\circ F(x)\in \mathbb{I})\\
	&=& F(x).
	\end{eqnarray*}
Therefore $f$ is a solution of \eqref{lcp} in $R_{a,0}[\mathbb{R}; \mathbb{I}]$. Hence by Propositions \ref{P3} and \ref{P6}, $g=\psi \circ f\circ \psi^{-1}$ is a solution of \eqref{lcp for g} in $S_{c, 0}[\mathbb{R}_+; \mathbb{J}]$.
Further, by Lemma \ref{L6}, $\phi_1$ and hence $g$ depends on
$n-1$ arbitrarily chosen orientation-preserving homeomorphisms $f_j:[x_j, x_{j-1}]\to [x_{j+1},x_j]$, $j=1,2,\ldots, n-1$,  where $x_0=b$.
\end{proof}

For the other class $ S_{d, \lambda}[\mathbb{R}_+; \mathbb{J}]$, we can similarly
prove the following result using Lemma \ref{L7}.

\begin{theorem} \label{T4}
		Let $\lambda \in [0,1)$ and $G\in S_{d, \lambda}[\mathbb{R}_+; \mathbb{J}]$ such that $\mathcal{R}(G)=\mathcal{R}(G|_\mathbb{J})$,  where $\mathbb{J}=[c,d|$. Then \eqref{lcp for g}
	has solutions in $S_{d, 0}[\mathbb{R}_+; \mathbb{J}]$. Moreover, each solution depends on $n-1$ arbitrarily chosen orientation-preserving homeomorphisms  $f_j:[x_{j-1},x_j]\to [x_j,x_{j+1}]$, $j=1, 2,\ldots,n-1$,  where
	 $x_0=a$ and $x_1, x_2, \ldots, x_n$ are given as in Lemma \ref{L7}.
\end{theorem}

In the special case that $\lambda=0$, \eqref{lcp for g} reduces to the equation
\begin{eqnarray}\label{irp}
g^n(x)=G(x),
\end{eqnarray}
i.e., the problem of iterative roots for a given function $G$. We have following results for solutions of \eqref{irp} on $\mathbb{R}_+$.

\begin{coro}\label{C2}
		Let $G$ be a continuous function on $\mathbb{R}_+$ such that $G$ is strictly increasing on $J$, $G(c)=c, G(d)< d, \mathcal{R}(G)=[c,G(d)]$  and  $G(x)<x$ for $x\in (c,d)$, where $J=[c,d]$.  Then \eqref{irp} has solutions on $\mathbb{R}_+$.
		Moreover, each solution depends on $n-1$ arbitrarily chosen orientation-preserving homeomorphisms $f_j:[x_j, x_{j-1}]\to [x_{j+1},x_j]$, $j=1,2,\ldots, n-1$,  where $x_0=b$ and $x_1, x_2, \ldots, x_n$ are given as in Lemma \ref{L6}.
\end{coro}
\begin{proof}
Follows from Theorem \ref{T3}, because $G\in S_{c,0}[\mathbb{R}_+; J]$ with $J=[c,d]$ such that $\mathcal{R}(F)=\mathcal{R}(F|_J)$.
\end{proof}

We have the following analogous result for the case $G(x)>x$, whose proof is similar.
\begin{coro}\label{C3}
		Let $G$ be a continuous function on $\mathbb{R}_+$ such that $G$ is strictly increasing on $\mathbb{J}$, $G(c)> c, G(d)= d, \mathcal{R}(G)=[G(c),d]$  and  $G(x)>x$ for $x\in (c,d)$, where $J=[c,d]$.  Then \eqref{irp} has solutions on $\mathbb{R}_+$.
		Moreover, each solution depends on $n-1$ arbitrarily chosen orientation-preserving homeomorphisms  $f_j:[x_{j-1},x_j]\to [x_j,x_{j+1}]$, $j=1, 2,\ldots,n-1$,  where
		$x_0=a$ and $x_1, x_2, \ldots, x_n$ are given as in Lemma \ref{L7}.
\end{coro}

Theorems \ref{T3} and \ref{T4} each give infinitely many solutions of \eqref{lcp for g} on $\mathbb{R}_+$ since infinitely many choices can be made for the initial function $f_1, f_2, \ldots, f_{n-1}$ in Lemmas \ref{L6} and \ref{L7}. Similar conclusions hold for Corollaries  \ref{C2} and \ref{C3}.

Next, we consider the case that $\lambda \le 0$.
%
In 2013,  assuming that $\lambda \le 0$, Zhang et al \cite{Zhang-Xu-Zhang} proved the existence of continuous solutions for \eqref{lcp}
on the compact $I$.
In what follows,
solving \eqref{lcp} with $\lambda \le 0$ on the whole $\mathbb{R}$, we obtain solutions of \eqref{lcp for g} on $\mathbb{R}_+$.
For compact intervals $I=[a,b]$ and $J=[c,d]$ of $\mathbb{R}$ and $\mathbb{R}_+$
respectively, and for $\lambda \in \mathbb{R}$, let
\begin{eqnarray*}
&&\mathcal{A}_{\lambda}[\mathbb{R}; I]:= \{f\in \mathcal{C}_b(\mathbb{R}): f|_I~\text{is strictly increasing}, f(a)=\lambda a~\text{and}~f(b)=\lambda b\},
\\
&&\mathcal{B}_{\lambda}[\mathbb{R}_+; J]:= \{g\in \mathcal{C}_b(\mathbb{R}_+): g|_J~\text{is strictly increasing}, g(c)=c^\lambda ~\text{and}~g(d)=d^\lambda \}.
\end{eqnarray*}

\begin{pro}\label{P7}
	Let $\lambda \in \mathbb{R}$. Then $g\in 	\mathcal{B}_{\lambda}[\mathbb{R}_+; J]$ if and only if $f=\psi^{-1} \circ g \circ \psi \in 	\mathcal{A}_{\lambda}[\mathbb{R}; I]$, where $\psi(x)=e^x$ and $I=\log(J)$.
\end{pro}

\begin{proof}
Let $g\in \mathcal{B}_{\lambda}[\mathbb{R}_+; J]$, where $\lambda \in \mathbb{R}$.
Since $f|_I=\psi^{-1}\circ (g|_J)\circ \psi$ and $g|_J$ is strictly increasing, clearly $f|_I$ is also strictly increasing. Also, $f(a)=\log g(e^a)=\log (g(c))=\log (c^\lambda)=\lambda \log c=\lambda a$ and similarly
$f(b)=\lambda b$. Hence $f\in 	\mathcal{A}_{\lambda}[\mathbb{R}; I]$. The converse follows similarly.
\end{proof}

\begin{lm}{\rm $($\cite{Zhang-Xu-Zhang}}, {\rm Corollary 1}$)$ \label{L8}
Let $\lambda \le 0$ and $F\in \mathcal{A}_{1-\lambda}[I; I]$, where $I=[a,b]$.
Then \eqref{lcp} has infinitely many solutions in $\mathcal{A}_{1}[I; I]$.
\end{lm}

\begin{rmk}\label{remark1}
{\rm	
The proof of the above lemma, seen in pp.82-89 of \cite{Zhang-Xu-Zhang},
shows steps to obtain those solutions:
\begin{description}
\item{\bf Step 1:}
 For each $\zeta, \xi \in (a,b)$ and $\lambda \le 0$,
let
\begin{eqnarray*}
\mathcal{A}_{\lambda}^\zeta[I]
&:=&\{f\in \mathcal{C}(I,\lambda I): f~\text{is strictly increasing on}~I, f(a)=\lambda a, f(b)=\lambda b,
\\
& &~~~~~~~~~~~~~~~~~~~~~ f(x)>\lambda x~\text{for}~x\in (a,b)~\text{and}~f~\text{is linear on}~[\zeta, b]\},
\\
\mathcal{B}_{\lambda}^\xi[I]
&:=&\{f\in \mathcal{C}(I,\lambda I): f~\text{is strictly increasing on}~I, f(a)=\lambda a, f(b)=\lambda b,
\\
& &~~~~~~~~~~~~~~~~~~~~~ f(x)<\lambda x~\text{for}~x\in (a,b)~\text{and}~f~\text{is linear on}~[a, \xi]\}.
\end{eqnarray*}
In this step, we construct solutions of \eqref{lcp} for $F\in \mathcal{A}_{1-\lambda}^\zeta[I]\cup\mathcal{B}_{1-\lambda}^\xi[I]$
	(see Theorem 1 in \cite{Zhang-Xu-Zhang}).
	This enables us to construct a sequence $(F_m)$ in $\mathcal{A}_{1-\lambda}^\zeta[I]\cup\mathcal{B}_{1-\lambda}^\xi[I]$ which converges to a given function $F$ of more general form and find the corresponding solutions $f_m$ for $m=1,2,\ldots$.

\item{\bf Step 2:}
Using the sequential compactness of $(f_m)$ and verifying that its limit $f$ is a solution of \eqref{lcp}, we arrive at the existence of solution of \eqref{lcp} for $F\in \mathcal{A}_{1-\lambda}[I]\cup\mathcal{B}_{1-\lambda}[I]$, where
\begin{eqnarray*}
\mathcal{A}_{\lambda}[I]
&:=&\{f\in \mathcal{C}(I,\lambda I): f~\text{is strictly increasing on}~I, f(a)=\lambda a,
\\
	& &~~~~~~~~~~~~~~~~~~~~ f(b)=\lambda b~\text{and}~f(x)>\lambda x~\text{for}~x\in (a,b)\},
\\
\mathcal{B}_{\lambda}[I]
&:=&\{f\in \mathcal{C}(I,\lambda I): f~\text{is strictly increasing on}~I, f(a)=\lambda a,
\\
& &~~~~~~~~~~~~~~~~~~~~~  f(b)=\lambda b~\text{and}~f(x)<\lambda x~\text{for}~x\in (a,b)\}
\end{eqnarray*}
for
 $\lambda \le 0$ (see Theorem 2 in \cite{Zhang-Xu-Zhang}).

\item{\bf Step 3:}
Dropping the assumption that location of $F$ is below or above the line $y=(1-\lambda)x$ made in $\mathcal{A}_{1-\lambda}[I]$ and
 $\mathcal{B}_{1-\lambda}[I]$, we obtain solutions of \eqref{lcp} for $F\in \mathcal{A}_{1-\lambda}[I; I]$ (see Corollary 1 in \cite{Zhang-Xu-Zhang}). In fact, given any $F\in \mathcal{A}_{1-\lambda}[I; I]$, let $\Gamma:=\{x\in I: F(x)=(1-\lambda)x\}$. Then $I=\Gamma \cup (\cup_jI_j)$ and $I_j$'s are disjoint open intervals, denoted by $(a_j, b_j)$'s, $a_j, b_j \in \Gamma$, such that $F(x)\ne (1-\lambda)x$ for $x\in (a_j, b_j)$. Then either $F_j\in \mathcal{B}_{1-\lambda}[I]$ or $F_j\in \mathcal{A}_{1-\lambda}[I]$, where $F_j:=F|_{I_j}$ for $j=1,2,\ldots$. By step 2, for each $j$ the equation
	\begin{eqnarray*}
	f^n(x)=\sum_{k=1}^{n-1}\lambda_k f^k(x)+F_j(x)
	\end{eqnarray*}
has a solution $f_j\in \mathcal{A}_1[I_j; I_j]$, which depends on the choice of a sequence $(F_{j,m})$ in $\mathcal{A}_{1-\lambda}^\zeta[I]\cup\mathcal{B}_{1-\lambda}^\xi[I]$. Then it follows that the function $f\in \mathcal{A}_1[I; I]$ defined by
	\begin{eqnarray*}
	f(x)=\left\{\begin{array}{cll}
	f_j(x)&\text{if}&x\in I_j,\\
	x&\text{if}&x\in \Gamma
	\end{array}\right.
	\end{eqnarray*}
is a solution of \eqref{lcp} on $I$.
\end{description}
Since infinitely many choices can be made for each of the sequences $(F_{j,m})$'s, Lemma \ref{L8} indeed gives infinitely many solutions of
\eqref{lcp} for $F\in \mathcal{A}_{1-\lambda}[I; I]$.
}
\end{rmk}

\begin{theorem}\label{T5}
		Let $\lambda \le 0$ and $G\in \mathcal{B}_{1-\lambda}[\mathbb{R}_+; J]$ such that $\mathcal{R}(G)=J^{1-\lambda}:=\{x^{1-\lambda}:x\in J\}$, where $J=[c,d]$. Then \eqref{lcp for g} has infinitely many solutions in $\mathcal{B}_{1}[\mathbb{R}_+; J]$. Moreover, each solution depends on the suitably chosen sequences  $(F_{j,m})$'s for $j=1,2,\ldots$
as indicated in the above Remark \ref{remark1}.
\end{theorem}

\begin{proof}
	Given $G\in \mathcal{B}_{1-\lambda}[\mathbb{R}_+; J]$, by Proposition \ref{P7}, we have $F=\psi^{-1}\circ G\circ \psi \in 	\mathcal{A}_{1-\lambda}[\mathbb{R}; I]$, where $\psi(x)=e^x$ and $I=\log(J)$. Also, since $\mathcal{R}(G)=J^{1-\lambda}$, we have $\mathcal{R}(F)=(1-\lambda)I$. So
	 $F_1:=F|_I\in \mathcal{A}_{1-\lambda}[I;I]$, and therefore by Lemma \ref{L8}, \eqref{lcp} has a solution $\phi_1$ in $\mathcal{A}_{1}[I;I]$. Let $f$ be the extension of $\phi_1$ to $\mathbb{R}$ as defined in \eqref{extension}. We prove that $f$ is a solution of \eqref{lcp} in $\mathcal{A}_{1}[\mathbb{R}_+; I]$.
Being a strictly increasing continuous map,
  $F_1:I \to (1-\lambda)I$ has the inverse $F_1^{-1}$, which is also strictly increasing and continuous on  $(1-\lambda)I$. Therefore,  as by assumption $\mathcal{R}(F)=\mathcal{R}(F_1)$, clearly $f$ is a well-defined map on $\mathbb{R}$. Also, $f$ is continuous on $\mathbb{R}$, being the composition of continuous maps $\phi_1, F_1^{-1}$ and $F$. Further,
		as $\phi_1 \in \mathcal{A}_{1}[I; I]$, it follows that $f|_I$ is strictly increasing, $f(a)=\lambda a$
		 and $f(b)=\lambda b$. Therefore $f\in \mathcal{A}_{1}[\mathbb{R}; I]$. Moreover, by a similar argument as in the proof of Theorem \ref{T3}, it can be shown that $f$ is a solution of \eqref{lcp} in $\mathcal{A}_{1}[\mathbb{R}; I]$. Hence by Propositions \ref{P3} and \ref{P7}, $g=\psi\circ f\circ \psi^{-1}$ is a solution of \eqref{lcp for g} in $\mathcal{B}_{1}[\mathbb{R}_+; J]$. Further, as indicated in Remark \ref{remark1}, construction of $\phi_1$ and hence that of $g$ depends on the choice of sequences  $(F_{j,m})$'s for $j=1,2,\ldots$.
\end{proof}

In the special case that $\lambda=0$, we have the following result for solutions of iterative root problem \eqref{irp} on $\mathbb{R}_+$.

\begin{coro}\label{C4}
	Let $G$ be a continuous function on $\mathbb{R}_+$ such that $G$ is strictly increasing on $J$, $G(c)=c$, $G(d)=d$ and $\mathcal{R}(G)=J$, where $J=[c,d]$. Then \eqref{irp} has infinitely many solutions on $\mathbb{R}_+$. Moreover, each solution depends on the suitably chosen sequences  $(F_{j,m})$'s for $j=1,2,\ldots$ as indicated in Remark \ref{remark1}.
\end{coro}

\begin{proof}
	Follows from Theorem \ref{T5}, because $G\in \mathcal{B}_1[\mathbb{R}_+; J]$ such that $\mathcal{R}(G)=J$.
\end{proof}

By comparing the coefficients of $g^k$, $1\le k \le n$, in equations \eqref{lcp for g} and \eqref{1}, we have $\alpha_k=-\lambda_k$ for $1\le k\le n-1$ and $\alpha_n=1$.
Further, if $\lambda_k \in \mathbb{Z}$ for $1\le k \le n-1$, then the assumption that $\sum_{k=1}^{n}\alpha_k$ is odd,
made in Proposition \ref{P4},
demands that $1-\sum_{k=1}^{n-1}\lambda_k$ is odd, i.e., $\lambda$ is even.
Thus,  using Proposition \ref{P4}, we can indeed extend solutions of \eqref{1} on $\mathbb{R}_+$ to  $\mathbb{R}_-$ whenever $\lambda_k \in \mathbb{Z}$
for all $1\le k \le n-1$ such that $\lambda$ is even.


\section{Examples and Remarks}



\begin{exmp}
{\rm 	Consider the equation
	\begin{eqnarray} \label{Ex4}
		(g(x))^\frac{3}{4} (g^{2}(x))^\frac{1}{4}=G(x),
	\end{eqnarray}
where $G:\mathbb{R}_+\to \mathbb{R}_+$ is defined by
	\begin{equation*}
G(x)=\left\{\begin{array}{cl}
1& {\rm if}~x\in (0,1],\\
e^{(1+\log x)\log \sqrt x}& {\rm if}~ x\in [1,e],\\
e& {\rm if}~x\in [e, \infty).
\end{array}\right.
\end{equation*}
Let $f(x):=\log g(e^x)$ and $F(x):=\log G(e^x)$ for $x\in \mathbb{R}$. Then \eqref{Ex4} reduces to the polynomial-like equation
		\begin{eqnarray*}\label{Ex3}
	\frac{3}{4}	f(x)+\frac{1}{4}f^2(x)=F(x),
	\end{eqnarray*}
	where $F:\mathbb{R}\to \mathbb{R}$ is the map defined by
	\begin{equation*}
	F(x)=\left\{\begin{array}{cl}
	0& {\rm if}~x\le 0,\\
	\frac{x^2+x}{2}& {\rm if}~ x\in [0,1],\\
		1& {\rm if}~x\ge 1.
	\end{array}\right.
	\end{equation*}
Note that $F\in \mathcal{F}_I(\frac{1}{2}, \frac{3}{2})$, where $I=[0,1]$.	Let $\delta=\frac{2}{3}$ and $M=2$. Then $K_1\delta=\frac{1}{2}$, $K_0M=\frac{3}{2}$, and therefore $F\in \mathcal{F}_I(K_1\delta, K_0M)$. This implies by Proposition \ref{P5} that $G\in \mathcal{G}_J(K_1\delta, K_0M)$, where $J=[1,e]$.  Also, $K_2=\frac{1}{4}<\frac{11}{12}=K_0$. Thus, all the hypotheses of Theorem \ref{existance} are satisfied. Hence \eqref{Ex4} has a unique solution $g$ in $\mathcal{G}_J(\frac{2}{3}, 2)$. }
\end{exmp}


\begin{exmp}
{\rm	Consider the equation
	\begin{eqnarray}\label{19}
\dfrac{(g^2(x))^3}{(g(x))^2}=G(x),
	\end{eqnarray}
	where $G:\mathbb{R}_+\to \mathbb{R}_+$ is defined by
	\begin{eqnarray*}
	G(x)=\left\{\begin{array}{cl}
	1& {\rm if}~x\in (0,1],\\
\sqrt[3]{x}	& {\rm if}~ x\in [1,e],\\
e^\frac{1}{3\log x}	& {\rm if}~x\in [e, \infty).
	\end{array}\right.
	\end{eqnarray*}
	Let $f(x):=\log g(e^x)$ and $F(x):=\log G(e^x)$ for $x\in \mathbb{R}$. Then
	\eqref{19} reduces to
$	-2f(x)+3f^2(x)=F(x)$,
	which is equivalent to
	 \begin{eqnarray*}\label{21}
	 f^2(x)-\frac{2}{3}f(x)=\frac{1}{3}F(x),
	 \end{eqnarray*}
	 where  $F:\mathbb{R}\to \mathbb{R}$ is the map defined by
	 	\begin{eqnarray*}
	 	F(x)=\left\{\begin{array}{cll}
	 		0&\text{if}&x\le 0,\\
	 		\frac{x}{3}&\text{if}&x\in [0,1],\\
	 		\frac{1}{3x}&\text{if}&x\ge 1.
	 	\end{array}
	 	\right.
	 \end{eqnarray*}
 Note that $H:=\frac{1}{3}F\in R_{0, \frac{2}{3}}[\mathbb{R}; I]$, where $I=[0,1]$.
  Therefore by Proposition \ref{P6},  it follows that the map $G_1$ defined by  $G_1(x)=e^{H(\log x )}$ lies in $\mathcal{S}_{1,\frac{2}{3}}[\mathbb{R}_+; J]$,
  where $J=[1,e]$. Also, since $\mathcal{R}(H)=[0,\frac{1}{9}]=\mathcal{R}(H|_I)$,
  we have $\mathcal{R}(G_1)=[1,\sqrt[9]{e}]=\mathcal{R}(G_1|_J)$. Therefore, by Theorem \ref{T3},
  	\begin{eqnarray*}
  \dfrac{(g^2(x))}{(g(x))^\frac{2}{3}}=G_1(x),
  \end{eqnarray*}
  and hence \eqref{19} has a solution $g$ on $\mathbb{R}_+$.}
\end{exmp}

\begin{exmp}
	{\rm 	Consider the equation
			\begin{eqnarray}\label{23}
		(g^2(x))^3(g(x))^6=G(x),
		\end{eqnarray}
	where $G:\mathbb{R}_+\to \mathbb{R}_+$ is defined by
\begin{eqnarray*}
	G(x)=\left\{\begin{array}{cl}
		1& {\rm if}~x\in (0,1],\\
		x^3	& {\rm if}~ x\in [1,2],\\
		\frac{7x+2}{x}	& {\rm if}~x\in [2, \infty).
	\end{array}\right.
\end{eqnarray*}	
Then $G\in \mathcal{B}_3[\mathbb{R}_+; J]$, where $J=[1,2]$. Also, $\mathcal{R}(G)=J^3$. Hence by Theorem \ref{T5}, \eqref{23} has a solution in $\mathcal{B}_1[\mathbb{R}_+; J]$.   }
\end{exmp}


We make the following observations regarding the two approaches (i.e.,
using fixed point theorem and
constructing solutions piece by piece) considered to solve \eqref{1}.
First,
the solutions $g$ of \eqref{1} obtained in Theorems \ref{T3} and \ref{T4} have exactly one fixed point at an end-point of $\mathcal{R}(g)$,
whereas each solution $g$ obtained in Theorems \ref{existance} and \ref{T5} has fixed points at
both end-points of $\mathcal{R}(g)$.
Second,
as noted before, using Theorem \ref{existance}, we cannot solve iterative root problem \eqref{irp}.
On the other hand, we can indeed obtain solutions of \eqref{irp} using Corollaries \ref{C2}, \ref{C3} and \ref{C4}.

Additionally, we remind that in section 4 we did not complete our discussion for all $\lambda\in \mathbb{R}$, because we have assumed that $0\le \lambda <1$ in Theorems \ref{T3} and \ref{T4}. Remark that theses theorems are not necessarily valid for $\lambda\ge 1$ and therefore our current approach cannot be used in this case to solve \eqref{1} on $\mathbb{R}_+$. More precisely, if $\lambda\ge 1$, then the sets $S_{c, \lambda}[\mathbb{R}_+; \mathbb{J}]$ and $S_{d, \lambda}[\mathbb{R}_+; \mathbb{J}]$ are not necessarily nonempty.
In fact,  if $G\in S_{1, 3}[\mathbb{R}_+; [1,2]]$, then by using the conditions {\bf (B1)} and {\bf (B2)}
we have $1<G(2)<{1}/{4}$, which is a contradiction.
We arrive at a similar contradiction that $1<G(1)<1$ if  $G\in S_{2, 3}[\mathbb{R}_+; [1.2]]$.
So, both the sets $S_{1, 3}[\mathbb{R}_+; [1.2]]$ and $S_{2, 3}[\mathbb{R}_+; [1.2]]$ are empty.

Further, as observed at the end of section \ref{fp method} (resp. \ref{construction}), we can extend the solutions of \eqref{1} (resp. \eqref{lcp for g}) on $\mathbb{R}_+$ to $\mathbb{R}_-$ whenever $\alpha_k\in \mathbb{Z}$ for $1\le k\le n$ such that $\sum_{k=1}^{n}\alpha_k$ is odd (resp. whenever $\lambda_k \in \mathbb{Z}$ for $1\le k \le n-1$ such that $\lambda$ is even). On the other hand, if
 $\alpha_k\in \mathbb{R}\setminus \mathbb{Z}$ for some $1\le k \le n$, then for any $G, g\in \mathcal{C}_b(\mathbb{R}_-)$,  $x \mapsto \prod_{k=1}^{n}(g^k(x))^{\alpha_k}$ is a  multi-valued complex map, whereas $x \mapsto G(x)$ is a  single valued real map. So, in order to obtain the equality $\prod_{k=1}^{n}(g^k(x))^{\alpha_k}=G(x)$, we have to choose branches of the complex logarithm  suitably, which not only depends on $x$ but also on each term of the product $\prod_{k=1}^{n}(g^k(x))^{\alpha_k}$.  Therefore,  solving \eqref{1} on $\mathbb{R}_-$ in this case is very difficult.
For a similar reason,
solving \eqref{lcp for g} on $\mathbb{R}_-$ is difficult if $\lambda_k\in \mathbb{R}\setminus \mathbb{Z}$ for some $1\le k\le n-1$.


\section*{Acknowledgment}
The authors are listed in alphabetic order of their names and their contributions are treated equally.
The author Murugan Veerapazham is supported by SERB, DST, Government of India, through
$ECR/2017/000765$.
The author Weinian Zhang is supported by NSFC \# 11831012, \# 11771307 and \# 11821001.



\end{document}